\documentclass[12pt]{amsart}
\usepackage{a4}
\usepackage{booktabs}
\usepackage[T1]{fontenc}
\usepackage[all]{xy}
\usepackage{lipsum}
\usepackage{url}
\usepackage{tikz-network}
\usepackage{stackrel}
\usepackage{color}

\usepackage{amsmath,amsthm,amssymb}
\newcommand{\aspas}[1]{``{#1}''}
\usepackage{xcolor}
\usepackage{graphicx}
\usepackage{tikz}
\usetikzlibrary{
  hobby,
  intersections,
  spath3,
  decorations.markings,
  arrows.meta,
}

\newtheorem{theorem}{Theorem}[section]
\newtheorem{lemma}[theorem]{Lemma}
\newtheorem{example}[theorem]{Example}
\newtheorem{proposition}[theorem]{Proposition}
\theoremstyle{definition}
\newtheorem{definition}[theorem]{Definition}

\newtheorem{remark}[theorem]{Remark}
\newtheorem{corollary}[theorem]{Corollary}

\numberwithin{equation}{section}

\begin{document}


\renewcommand{\bf}{\bfseries}
\renewcommand{\sc}{\scshape}

\title[Injective hom-complexity between groups]%
{Injective hom-complexity between groups}

\author[C. A. Ipanaque Zapata]{Cesar A. Ipanaque Zapata$^1$}
\address[C. A. Ipanaque Zapata]{Departamento de Matem\'atica, IME-Universidade de S\~ao Paulo, Rua do Mat\~ao, 1010 CEP: 05508-090, S\~ao Paulo, SP, Brazil}
\email{cesarzapata@usp.br}
\thanks {$^1$Corresponding author: cesarzapata@usp.br}

\author[M.O. Gonzales Bohorquez]{Martha O. Gonzales Bohorquez}%
\address[M.O. Gonzales Bohorquez]{ Facultad de Ciencias Matemáticas, Universidad Nacional Mayor de San Marcos,   Lima, Perú}
\email{mgonzalesb@unmsm.edu.pe}




\begin{abstract} We present the notion of injective hom-complexity, leading to a connection between the covering number of a group and the sectional number of a group homomorphism, and provide estimates for computing this invariant.
\end{abstract}
\maketitle


\section{Introduction}\label{secintro}%
In this article, the term \aspas{homomorphism} refers to a group homomorphism. The symbol $G\to H$ means that there is a homomorphism from $G$ to $H$; otherwise, we write $G\not\rightarrow H$. The symbols $0$ or $1$ represent the trivial group. We write $\text{ord}(g)$ to refer to the order of the element $g$ and $|G|$ to refer to the order of a group $G$. The symbol $\lceil m\rceil$ denotes the least integer greater than or equal to $m$, while $\lfloor m\rfloor$ denotes the greatest integer less than or equal to $m$. We write the set $[k]=\{1,\ldots,k\}$.    
 
\medskip Given two groups $G$ and $H$, it is natural to pose the following question: Is there an injective homomorphism $G\to H$? The answer is yes if and only if $H$ has a copy of $G$ as a subgroup, that is, there exists a subgroup $K$ of $H$ such that $K$ is isomorphic to $G$.

\medskip Motivated by this question, we introduce the notion of injective hom-complexity between two groups $G$ and $H$, denoted by $\text{IC}(G;H)$ (Definition~\ref{defn:complexity}), along with its basic results. More precisely, $\text{IC}(G;H)$ is defined as the least positive integer $\ell$ such that there are $\ell$ distinct subgroups $G_j$ of $G$ with $G=G_1\cup\cdots\cup G_\ell$, and over each $G_j$, there exists an injective homomorphism $G_j\to H$. For instance, we have $\text{IC}(G;H)=1$ if and only if there is an injective homomorphism $G\to H$. 

\medskip From \cite[p. 492]{harver1959} (see also \cite[p. 1071]{rosenfeld1963}, \cite[p. 44]{cohn1994}), given a group $G$, the \textit{covering number} of $G$, denoted by $\sigma(G)$, is the least positive integer $\ell$ such that there are $\ell$ distinct proper subgroups $G_j$ of $G$ with $G=G_1\cup\cdots\cup G_\ell$.

\medskip We discuss a connection between the injective hom-complexity $\text{IC}(G;H)$ and the covering number $\sigma(G)$. For instance, we have $\sigma(G)\leq \text{IC}(G;H)$ whenever $H$ has not a copy of $G$ as a subgroup (Lemma~\ref{thm:lower-bound}).

\medskip The main results of this paper are as follows:
\begin{itemize}
    \item Given three groups $G,H$, and $K$, we present a relation between the injective hom-complexities $\mathrm{IC}(G;H), \mathrm{IC}(H;K)$, and $\mathrm{IC}(G;K)$ (Theorem~\ref{thm:inequality-three-graphs}). In particular, this shows that injective hom-complexity is a group invariant (Corollary~\ref{cor:invariant-iso-complexity}).
    \item We establish a general upper bound (Theorem~\ref{thm:general-upper-bound}).
    \item We present a formula for $\mathrm{IC}(G;H)$ whenever $H$ is a finite group having only cyclic proper subgroups of prime order (Theorem~\ref{thm:to-Zp}). 
    \item Sub-additivity (Theorem~\ref{thm:category-union}).
    \item We compare the injective hom-complexity of a product in terms of the injective hom-complexity of its factors (Theorem~\ref{thm:product-inequality-complexity}). 
\end{itemize}

\medskip This paper is organized into two sections: In Section~\ref{sec:one}, we introduce the notion of injective hom-complexity $\text{IC}(G;H)$ for two groups $G$ and $H$ (Definition~\ref{defn:complexity}) along with its basic properties. Section~\ref{sec:appli} presents new insights into group theory. We close this section with Remark~\ref{rem:complexity-sectionalnumber}, which presents a direct connection between the injective hom-complexity and the notion of sectional number.

\section{Injective hom-complexity}\label{sec:one}
In this section, we introduce the notion of injective hom-complexity along with its basic properties. Several examples are provided to support this theory. 

\subsection{Definition and Examples} 
We present the main definition of this work.

\begin{definition}[Injective hom-complexity]\label{defn:complexity}
Let $G$ and $H$ be groups. The \textit{injective hom-complexity} from $G$ to $H$, denoted by $\text{IC}(G;H)$, is the least positive integer $k$ such that there exist subgroups $G_1,\ldots,G_k$ of $G$ satisfying $G=G_1\cup\cdots\cup G_k$, and for each $G_i$, there exists an injective homomorphism $f_i:G_i\to H$. We set $\text{IC}(G;H)=\infty$ if no such integer $k$ exists.
\end{definition}

A collection $\mathcal{M}=\{f_i:G_i\to H\}_{i=1}^\ell$, where $G_1,\ldots,G_\ell$ are subgroups of $G$ such that $G=G_1\cup\cdots\cup G_\ell$ and  each $f_i:G_i\to H$ is an injective homomorphism, is called an \textit{injective quasi-homomorphism} from $G$ to $H$. An injective quasi-homomorphism $\mathcal{M}=\{f_i:G_i\to H\}_{i=1}^\ell$ is termed \textit{optimal} if $\ell=\text{IC}(G;H)$. Observe that a unitary injective quasi-homomorphism $\{f:G\to H\}$ is optimal and constitutes an injective homomorphism from $G$ to $H$. Additionally, any injective quasi-homomorphism $\mathcal{M}=\{f_i:G_i\to H\}_{i=1}^\ell$ induces a map $f:G\to H$ defined by  $f(g)=f_i(g)$, where $i$ is the least index such that $g\in G_i$. 

\medskip By Definition~\ref{defn:complexity}, we can make the following remark. 

\begin{remark}\label{rem:def-ob}
\noindent \begin{itemize}
    \item[(1)]  $\text{IC}(G;H)=1$  if and only if there exists an injective homomorphism $G\to H$, which is equivalent to saying that $H$ admits a copy of $G$ as a subgroup. 
    \item[(2)] $\text{IC}(G;H)=\infty$ whenever $G$ is an infinite group and $H$ is a finite group. In fact, suppose that $\text{IC}(G;H)=k<\infty$ and consider subgroups $G_1,\ldots,G_k$ of $G$ satisfying $G=G_1\cup\cdots\cup G_k$, and for each $G_i$, there exists an injective homomorphism $f_i:G_i\to H$. Since $H$ is finite, each $G_i$ is finite, and thus $G$ is finite, which leads to a contradiction. 
    \item[(3)] If $G$ is a cyclic group, then \[\text{IC}(G;H)=\begin{cases}
        1,&\hbox{ if $H$ has a copy of $G$ as a subgroup;}\\
         \infty,&\hbox{ if $H$ has not a copy of $G$ as a subgroup.}
    \end{cases}\]  
    \item[(4)] Let $G$ and $H$ be any groups. If $\mathrm{IC}(G;H)<\infty$, then for each element $a$ of $G$ there exists an element $b$ of $H$ such that $\text{ord}(b)=\text{ord}(a)$. In fact, let $k=\mathrm{IC}(G;H)$, and consider subgroups $G_1,\ldots,G_k$ of $G$ satisfying $G=G_1\cup\cdots\cup G_k$, and for each $G_i$, there exists an injective homomorphism $f_i:G_i\to H$ (i.e., $H$ has a copy of $G_i$ as a subgroup). Let $a\in G$. Then, $a\in G_i$ for some $i\in\{1,\ldots,k\}$. Let $b:=f(a)\in H$. Since $f_i$ is injective, $\text{ord}(b)=\text{ord}(a)$. 
    \item[(5)] If $\mathcal{M}=\{f_i:G_i\to H\}_{i=1}^\ell$ is an optimal nonunitary injective quasi-homomorphism (i.e., $\ell=\mathrm{IC}(G;H)>1$), then \begin{enumerate}
        \item[(i)] For each $i=1,\ldots,\ell$, we have $\bigcup_{\overset{j=1}{j\neq i}}^{\ell} G_j\subsetneq G$. In particular, $G_i\not\subseteq G_j$ for any $1\leq i\neq j\leq\ell$. 
        \item[(ii)] For any $1\leq i\neq j\leq\ell$, $\langle G_i\cup G_j\rangle= G$ or there is not an injective homomorphism $\langle G_i\cup G_j\rangle\to H$. 
    \end{enumerate}  
\end{itemize}
\end{remark}

Observe that the other implication of Remark~\ref{rem:def-ob}(4) does not hold in general. For instance, consider the group $G=C_2^{\infty}=\{(\alpha_n)_{n\geq 1}:~\alpha_n\in C_2 \text{ for all $n\geq 1$}\}$ with the component-wise operation, i.e., $(\alpha_n)_{n\geq 1}+(\alpha'_n)_{n\geq 1}=(\alpha_n+\alpha'_n)_{n\geq 1}$, and $H=C_2$. Note that each element of $C_2^{\infty}$ is of order $2$. While $\mathrm{IC}(C_2^{\infty};C_2)=\infty$ (see Remark~\ref{rem:def-ob}(2)). We will see, in Theorem~\ref{thm:general-upper-bound}, that the other implication of Remark~\ref{rem:def-ob}(4) holds whenever $G$ is finite.

\medskip We have the following example.

\begin{example}\label{exam:complexity-less-injective}
    Let $G$ and $H$ be groups. We have \begin{enumerate}
    \item[(1)] $\mathrm{IC}(0;H)=1$.
        \item[(2)] $\mathrm{IC}(G;0)=\begin{cases}
            1,&\hbox{ if $G=0$;}\\
            \infty,&\hbox{ if $G\neq 0$.}
        \end{cases}$
        \item[(3)] $\mathrm{IC}(C_2;H)=\begin{cases}
            1,&\hbox{ if $H$ has an element of order 2;}\\
            \infty,&\hbox{ if $H$ has not an element of order 2.}
        \end{cases}$
        \item[(4)] $\mathrm{IC}(G;C_2)=\begin{cases}
            1,&\hbox{ if $G=0$ or $C_2$;}\\
            \infty,&\hbox{ if $G$ is an infinite group or has an element }\\
            &\hbox{ of order at least 3.}
        \end{cases}$ 
    \end{enumerate}
\end{example}



\subsection{Triangular Inequality and Group Invariant}
Given a homomorphism $f:G\to H$ and a subgroup $K$ of $H$, the \textit{image inverse} of $K$ through $f$, $f^{-1}(K)$, is a subgroup of $G$. Note that the restriction map $f_|:f^{-1}(K)\to K$ is a homomorphism, called the \textit{restriction homomorphism}. 

\medskip Given three groups $G,H$, and $K$, there is a relation between the injective hom-complexities $\mathrm{IC}(G;H), \mathrm{IC}(H;K)$, and $\mathrm{IC}(G;K)$. 

\begin{theorem}[Triangular Inequality]\label{thm:inequality-three-graphs}
  Let $G,H$, and $K$ be groups. Then, \[\mathrm{IC}(G;K)\leq \mathrm{IC}(G;H)\cdot \mathrm{IC}(H;K).\] In particular, if there exists an injective homomorphism $G'\to G$, then \[\mathrm{IC}(G';H)\leq\mathrm{IC}(G;H)\] for any group $H$. Similarly, if there exists an injective homomorphism $H'\to H$, then $\mathrm{IC}(G;H)\leq \mathrm{IC}(G;H')$ for any group $G$.
  \end{theorem}
\begin{proof}
  Let $m=\mathrm{IC}(G;H)$ and $n=\mathrm{IC}(H;K)$. Let $\mathcal{M}_1=\{g_{i}:G_{i}\to H\}_{i=1}^{m}$ be an optimal injective quasi-homomorphism from $G$ to $H$, and $\mathcal{M}_2=\{h_{j}:H_{j}\to K\}_{j=1}^{n}$ be an optimal injective quasi-homomorphisms from $H$ to $K$. Define  $G_{i,j}:=g_{i}^{-1}(H_j)$ for each $i\in\{1,\ldots,m\}$ and each $j\in\{1,\ldots,n\}$. We have $G=\bigcup_{i,j=1}^{m,n} G_{i,j}$. Note that  $G_{i,j}\neq\varnothing$ is a subgraph of $G_i$ (and consequently a subgraph of $G$). We also consider the restriction homomorphism $(f_i)_|:G_{i,j}\to H_j$. This leads to the composition \[G_{i,j}\stackrel{(f_i)_|}{\to} H_j\stackrel{h_j}{\to} K.\] Since $g_{i}$ and $h_{j}$ are injective, the composition $G_{i,j}\stackrel{(f_i)_|}{\to} H_j\stackrel{h_j}{\to} K$ is also injective. Therefore, we obtain $\mathrm{IC}(G;K)\leq m\cdot n=\mathrm{IC}(G;H)\cdot \mathrm{IC}(H;K)$.  
\end{proof}

The inequality in Theorem~\ref{thm:inequality-three-graphs} is sharp. For instance, consider $K=H$; then $\mathrm{IC}(G;H)= \mathrm{IC}(G;H)\cdot \mathrm{IC}(H;H)$. 

\medskip From Theorem~\ref{thm:inequality-three-graphs}, we observe that if $G'\to G$ and $G\to G'$ are injective, then $\mathrm{IC}(G';H)=\mathrm{IC}(G;H)$ for any group $H$. Similarly, if $H'\to H$ and $H\to H'$ are injective, then $\mathrm{IC}(G;H')=\mathrm{C}(IG;H)$ for any group $G$. In particular, this shows  that injective hom-complexity is a group invariant, meaning it is preserved under group isomorphisms.

\begin{corollary}[Group Invariant]\label{cor:invariant-iso-complexity}
    If $G'$ is isomorphic to $G$ and $H'$ is isomorphic to $H$, then \[\mathrm{IC}(G;H)=\mathrm{IC}(G';H').\]
\end{corollary}

\subsection{Lower Bound} From \cite[p. 492]{harver1959} (see also \cite[p. 1071]{rosenfeld1963}, \cite[p. 44]{cohn1994}), given a group $G$, the \textit{covering number} of $G$, denoted by $\sigma(G)$, is the least positive integer $\ell$ such that there are $\ell$ distinct proper subgroups $G_j$ of $G$ with $G=G_1\cup\cdots\cup G_\ell$. We set $\sigma(G)=\infty$ if no such $\ell$ exists. Observe that $\sigma(G)\geq 3$ for any group $G$ \cite[Theorem 1, p. 492]{harver1959}.

\medskip We have the following lower bound for the injective hom-complexity. 

\begin{lemma}[Lower Bound]\label{thm:lower-bound}
 Let $G$ and $H$ be groups such that $H$ does not have a copy of $G$ as a subgroup. The inequality \[\sigma(G)\leq \mathrm{IC}(G;H)\] holds. 
\end{lemma}
\begin{proof}
   Since $H$ does not have a copy of $G$ as a subgroup, $\mathrm{IC}(G;H)>1$. If $\mathrm{IC}(G;H)=\infty$, then the inequality $\sigma(G)\leq \mathrm{IC}(G;H)$ always hold. For the case $\mathrm{IC}(G;H)<\infty$, let $n=\mathrm{IC}(G;H)$ and consider $G_1,\ldots, G_n$ subgroups of $G$ such that $G=G_1\cup\ldots\cup G_n$ and for each $G_j$, there exists an injective homomorphism $G_j\to H$. Since $n>1$, each $G_j$ is a proper subgroup of $G$. Then, $\sigma(G)\leq n=\mathrm{IC}(G;H)$.  
\end{proof}

Lemma~\ref{thm:lower-bound} implies the following example.

\begin{example}
  \noindent
    \begin{enumerate}
        \item[(1)] Let $C$ be a cyclic group. Since $\sigma(C)=\infty$ \cite[p. 491]{harver1959}, \[\mathrm{IC}(C;H)=\infty\] for any group $H$ that does not have a copy of $C$ as a subgroup.
         \item[(1)] Since $\sigma(\mathbb{Q})=\infty$ \cite[p. 491]{harver1959}, \cite[p. 29]{zapata2024}, \[\mathrm{IC}(\mathbb{Q};H)=\infty\] for any group $H$ that does not have a copy of $\mathbb{Q}$ as a subgroup.
    \end{enumerate}
\end{example}


\subsection{Upper Bound} 
 Before to present an upper for the injective hom-complexity, we recall the notion of cyclic covering number given in (\cite[Definition 2.18]{zapata2025}, cf. after of \cite[Example 3.12]{zapata2024}).

\medskip   Let $G$ be a group. The \textit{cyclic covering number} of $G$, denoted by $\sigma_c(G)$, is the least positive integer $m$ such that there exist cyclic proper subgroups $C_1,\ldots,C_m$ of $G$ such that $G=C_1\cup\cdots\cup C_m$. We set $\sigma_c(G)=\infty$ if no such $m$ exists. Observe that $\sigma_c(G)\geq\sigma(G)\geq 3$. 

\medskip For instance, we have $\sigma_c(\mathbb{Z}\times C_2)=\sigma(\mathbb{Z}\times C_2)=3$ because $\mathbb{Z}\times C_2=\langle (1,\overline{0})\rangle\cup \langle (1,\overline{1})\rangle\cup \langle (0,\overline{1})\rangle$.

\medskip Let $G$ be a finite noncyclic group. Every (proper) cyclic subgroup $\langle x\rangle$ of $G$ has $\varphi(\text{ord}(x))$ generators. For each $x\in G$, the number of distinct cyclic subgroups of order $\text{ord}(x)$ is given by \[\dfrac{\text{the number of distinct elements of order $\text{ord}(x)$}}{\varphi(\text{ord}(x))},\] where $\varphi$ is the Euler's totient function. Hence, the number of distinct nontrivial proper cyclic subgroups of $G$ is $\sum_{\overset{x\in G}{x\neq 1}}\dfrac{1}{\varphi(\text{ord}(x))}$, and thus \begin{eqnarray}\label{eqn:ciclic-number}
    \sigma_c(G)&\leq\displaystyle\sum_{\overset{x\in G}{x\neq 1}}\dfrac{1}{\varphi(\text{ord}(x))}.
\end{eqnarray}  

\medskip Inequality~(\ref{eqn:ciclic-number}) can be strict (see \cite[Example 2.19(3)]{zapata2025}). We have the following example.

\begin{example}\cite[Example 2.19]{zapata2025}\label{exam:cyclic-z2z2}
    \noindent \begin{enumerate}
        \item[(1)] Let $G$ be a finite noncyclic \textit{elementary $p$-group} for some prime $p$, i.e., all the non-identity elements of $G$ have the same order $p$. In this case, $\varphi(\text{ord}(x))=\varphi(p)=p-1$ for any $x\in G$, $x\neq 1$. In addition, if $\langle x\rangle\subseteq \langle y\rangle$ for some $x,y\in G\setminus\{1\}$, then $\langle x\rangle= \langle y\rangle$. Hence, \[\sigma_c(G)= \dfrac{|G|-1}{p-1}.\] 
        \item[(2)] Let $n\geq 2$ be an integer and $p$ be a prime. Let $C_p^n=C_p\times\cdots\times C_p$ ($n$ times). By Item (1), we have \[\sigma_{c}\left(C_p^n\right)=\dfrac{p^n-1}{p-1}.\] On the other hand, the equality $\sigma\left(C_p^n\right)=p+1$ holds \cite[Theorem 2, p. 45]{cohn1994}, \cite[Theorem, p. 1071]{rosenfeld1963}. For instance, $\sigma_{c}\left(C_p\times C_p\right)=\sigma\left(C_p\times C_p\right)=p+1$.
    \end{enumerate} 
\end{example}

Now, we have the following upper bound for the IC.

\begin{theorem}[Upper Bound]\label{thm:general-upper-bound}
Let $G$ and $H$ be groups such that for each element $a$ of $G$, there exists an element $b$ of $H$ for which $\mathrm{ord}(b)=\mathrm{ord}(a)$. We have \[\mathrm{IC}(G;H)\leq \sigma_c(G).\]      
\end{theorem}
\begin{proof}    
  If $\sigma_c(G)=\infty$, then the inequality $\mathrm{IC}(G;H)\leq \sigma_c(G)$ always holds. Now, suppose $\sigma_c(G)=k<\infty$ and consider $\{C_1,\ldots,C_k\}$ a collection of cyclic proper subgroups of $G$ such that $G=C_1\cup\cdots\cup C_k$. Let $C_j=\langle a_j\rangle$ for each $j=1,\ldots,k$. For each $j=1,\ldots,k$, there exists $b_j\in H$ such that $\mathrm{ord}(b_j)=\mathrm{ord}(a_j)$. The map $h:\langle a_j\rangle\to H$ given by $h(a_j^m)=b_j^m$ for any $m\in\mathbb{Z}$ is an injective homomorphism. Therefore, $\mathrm{IC}(G;H)\leq k=\sigma_c(G)$.
\end{proof}

Theorem~\ref{thm:general-upper-bound} together with the inequality $\mathrm{IC}(G;H)\geq \sigma(G)$ (see Lemma~\ref{thm:lower-bound}) implies the following result. 

\begin{corollary}\label{prop:cov-equal-cyclic}
 Let $G$ and $H$ be groups such that $H$ does not admit a copy of $G$, and for each element $a$ of $G$ there exists an element $b$ of $H$ for which $\mathrm{ord}(b)=\mathrm{ord}(a)$. Then \[\sigma(G)\leq\mathrm{IC}(G;H)\leq\sigma_c(G).\] In particular, if $\sigma(G)=\sigma_c(G)$, then \[\mathrm{IC}(G;H)=\sigma_c(G)=\sigma(G).\]     
\end{corollary}

We have the following example.

\begin{example}\label{exam:proper-sec-number-codomain-z2z2}
 \noindent  \begin{enumerate}
      \item[(1)] Let $p$ be a prime number and $H$ be a group admitting an element of order $p$ such that $\mathrm{IC}(C_p\times C_p;H)>1$. By Example~\ref{exam:cyclic-z2z2}(2), $\sigma_{c}(C_p\times C_p)=\sigma(C_p\times C_p)=p+1$. Hence, by Corollary~\ref{prop:cov-equal-cyclic}, \[\mathrm{IC}(C_p\times C_p;H)= p+1.\] 
      \item[(2)] Let $H$ be a group admitting an element of order $3$ such that $\mathrm{IC}(C_3\times C_3\times C_3;H)>1$. Observe that $\sigma_{c}(C_3\times C_3\times C_3)=13$ and $\sigma(C_3\times C_3\times C_3)=4$ (see Example~\ref{exam:cyclic-z2z2}(2)). By Corollary~\ref{prop:cov-equal-cyclic}, \[4\leq \mathrm{IC}(C_3\times C_3\times C_3;H)\leq 13.\] We will see, in Example~\ref{exam:codomain-zp}, there exists a group $H$ such that $\mathrm{IC}(C_3\times C_3\times C_3;H)=13$.
  \end{enumerate} 
\end{example}

On the other hand, Theorem~\ref{thm:general-upper-bound} also implies the following result.

\begin{corollary}\label{cor:ic-ciclic-equal}
 Let $G$ and $H$ be groups such that $H$ does not admit a copy of $G$, and for each element $a$ of $G$ there exists an element $b$ of $H$ for which $\mathrm{ord}(b)=\mathrm{ord}(a)$. If any proper subgroup of $H$ is cyclic, then \[\mathrm{IC}(G;H)=\sigma_c(G).\]   
\end{corollary}

\begin{remark}
   Finite groups whose only proper subgroups are cyclic are fully classified \cite{miller-moreno1903}. These are called minimal non-cyclic groups or sometimes Miller–Moreno groups. A finite group $G$ has only cyclic proper subgroups if and only if $G$ is one of the following:
\begin{itemize}
    \item Cyclic groups $C_n$.
    \item Generalized quaternion groups $Q_{2^n}$, $n\geq 3$, of order $2^n$.
    \item Non-abelian groups $C_q\rtimes C_p$ of order $pq$ (with $p,q$ primes, $p<q$, $p|q-1$).
\end{itemize}

In particular, a finite group $G$ has only cyclic proper subgroups of prime order if and only if $G$ is one of the following:
\begin{itemize}
    \item $C_{p^2}$.
    \item $C_p\times C_p$.
    \item Any group of order $pq$ (cyclic or the non-abelian semidirect product $C_q\rtimes C_p$) with $p,q$ primes, $p<q$.
\end{itemize} 
\end{remark}

We recall the famous Poincaré's formula or inclusion-exclusion formula. Given finite sets $A_1,\ldots,A_k$ with $k\geq 1$, the following equality \[|A_1\cup\cdots\cup A_k|=\sum_{n=1}^{k}(-1)^{n+1}\sum_{\overset{J\subseteq [k]}{|J|=n}}|\bigcap_{j\in J} A_j|\] always holds. 

\medskip We have the following statement.

\begin{theorem}\label{thm:to-Zp}
  Let $G$ be a nontrivial finite group.  \begin{enumerate}
      \item[(1)] If $\mathrm{IC}(G;C_p)<\infty$ with $p$ a prime number, then \[|G|=\mathrm{IC}(G;C_p)(p-1)+1 \quad \text{ and } \quad p~|~\mathrm{IC}(G;C_p)-1.\]  
      \item[(2)] If $G$ is not abelian of order $pq$ with $p,q$ primes, $p<q$, then \[\mathrm{IC}(G;C_{pq})=q+1.\] 
      \item[(3)] If $G=C_p\times C_p$ with $p$ a prime number, then  \[\mathrm{IC}(C_p\times C_p;C_{p^2})=p+1.\] 
  \end{enumerate}  
\end{theorem}
\begin{proof}
 \noindent\begin{enumerate}
     \item[(1)]   Let $k=\mathrm{IC}(G;C_p)<\infty$, and consider nontrivial subgroups $G_1,\ldots,G_k$ of $G$ satisfying $G=G_1\cup\cdots\cup G_k$, and for each $G_i$, there exists an injective homomorphism $f_i:G_i\to C_p$ (i.e., $C_p$ has a copy of $G_i$ as a nontrivial subgroup). Since the only subgroups of $C_p$ are the trivial group and itself, each $f_i$ is an isomorphism, and thus, the only subgroups of each $G_i$ are the trivial group and itself. Then, for each $J\subseteq [k]$ with $|J|\geq 2$, observe that $\bigcap_{j\in J} G_j=1$, is the trivial group. Otherwise, $\bigcap_{j\in J} G_j=G_i$ for any $i\in J$. In particular, there exist distinct $i,i'\in J$ such that $G_i\subseteq G_{i'}$ (and of course, $G_i\cup G_{i'}=G_{i'}$), and hence $\mathrm{IC}(G;C_p)<k$, which leads to a contradiction.

    By the famous Poincaré's formula, we have:
    \begin{align*}
        |G|&=\sum_{n=1}^{k}(-1)^{n+1}\sum_{\overset{J\subseteq [k]}{|J|=n}}|\bigcap_{j\in J} G_j|\\
        &=\sum_{j=1}^{k}|G_j|+\sum_{n=2}^{k}(-1)^{n+1}\sum_{\overset{J\subseteq [k]}{|J|=n}}|\bigcap_{j\in J} G_j|\\
        &=\sum_{j=1}^{k}p+\sum_{n=2}^{k}(-1)^{n+1}\sum_{\overset{J\subseteq [k]}{|J|=n}}1\\
        &=kp+\sum_{n=2}^{k}(-1)^{n+1}\displaystyle{k \choose n},
    \end{align*} where $\displaystyle{k \choose n}=\dfrac{k!}{n!(k-n)!}$, is the binomial coefficient, i.e., it is the number of different subsets of $n$ elements that can be chosen from $[k]$. On the other hand, by the famous binomial theorem, we have \[\sum_{n=2}^{k}(-1)^{n+1}\displaystyle{k \choose n}=1-k.\] Therefore, $|G|=kp+1-k=\mathrm{IC}(G;C_p)(p-1)+1$. 
    
    Since $G_i$ is a subgroup of $G$ with $|G_i|=p$, then by the famous Lagrange's theorem, $p~|~|G|$, and thus $p~|~\mathrm{IC}(G;C_p)-1$. 
    
    \item[(2)] Let $k:=\mathrm{IC}(G;C_{pq})<\infty$, and consider nontrivial subgroups $G_1,\ldots,G_k$ of $G$ satisfying $G=G_1\cup\cdots\cup G_k$ for which for each $G_i$, there exists an injective homomorphism $f_i:G_i\to C_{pq}$. Then $C_{pq}$ has a copy of $G_i$ as a nontrivial proper subgroup. Since the only proper subgroups of $C_{pq}$ are the trivial group, a copy of $C_p$, and a copy of $C_{q}$, the only subgroups of each $G_i$ are the trivial group and itself. Then, for each $J\subseteq [k]$ with $|J|\geq 2$, we have $\bigcap_{j\in J} G_j=1$ as the trivial group. Furthermore, by Sylow's theorem, there exists a unique $i_0\in \{1,\ldots,k\}$ such that $|G_{i_0}|=q$ (because $G$ admits a unique $q$-Sylow subgroup).
    
  Similarly, as in the proof of Item (1), by the Poincaré's formula and the binomial theorem, we obtain \begin{align*}
   pq&=  |G|\\
   &=(k-1)p+q+1-k\\
   & =k(p-1)+q-p+1\\
     & =\mathrm{IC}(G;C_{pq})(p-1)+q-p+1. 
  \end{align*} and thus $\mathrm{IC}(G;C_{pq})=q+1$.
  \item[(3)] It is similar to Items (1) and (2). 
    
 \end{enumerate} 
\end{proof}

Let $G$ be a nontrivial finite group and $p$ be a prime number. By Remark~\ref{rem:def-ob}(4) and Theorem~\ref{thm:general-upper-bound} we have $\mathrm{IC}(G;C_p)<\infty$ if and only if the order of any nontrivial element of $G$ is $p$. Hence, we have the following example.

\begin{example}\label{exam:codomain-zp}
 Let $p$ be a prime number and $n\geq 1$. For $G=C_p^n$, the $n$th direct product of $C_p$, observe that the order of any nontrivial element of $C_p^n$ is $p$. Hence, by Theorem~\ref{thm:to-Zp}(1), we have \[\mathrm{IC}(C_p^n;C_p)=(p^n-1)/(p-1).\] For the case $n\geq 2$, we also obtain it using Example~\ref{exam:cyclic-z2z2}(2) together with Corollary~\ref{cor:ic-ciclic-equal}.
 
 On the other hand, recall that $\sigma(C_p^n)=p+1$ for any $n\geq 2$ \cite[Theorem 2, p. 45]{cohn1994}, \cite[Theorem, p. 1071]{rosenfeld1963}. Hence, \[\mathrm{IC}(C_p^n;C_p)-\sigma(C_p^n)=(p^n-1)/(p-1)-p-1\] for any $n\geq 2$. In particular, it shows that the difference $\mathrm{IC}(G;C_p)-\sigma(G)$ can be arbitrarily large.  
\end{example}

\medskip Moreover, Theorem~\ref{thm:to-Zp}(2) implies the following example.

\begin{example}
For each $n\geq 3$, we have the \textit{$n$-th dihedral group} \[D_n=C_n\rtimes C_2=\langle r,a:~r^n=1,a^2=1, ara=r^{-1}\rangle.\] By extension it is given by \[D_n=\{1,r,r^2,\ldots,r^{n-1},a,ra,r^2a,\ldots,r^{n-1}a\}.\] 
  
  We consider the case $n=p\geq 3$ a prime number. Hence, by Theorem~\ref{thm:to-Zp}(2), we have \[\mathrm{IC}(D_p;C_{2p})=p+1.\]     
\end{example}

\section{Applications}\label{sec:appli}
In this section, we present several applications motivated by the categorical viewpoint. 
\subsection{Sub-additivity} The following statement demonstrates the sub-additivity property of injective hom-complexity. 

\medskip Recall that any group $G$ cannot be the union of 2 proper subgroups. 

\begin{theorem}[Sub-additivity]\label{thm:category-union}
    Let $G,H$ be groups, and let $A,B,C$ be proper subgroups of $G$ such that $G=A\cup B\cup C$. Then: \[\max\{\mathrm{IC}(A;H),\mathrm{IC}(B;H),\mathrm{IC}(C;H)\}\leq \mathrm{IC}(G;H)\leq \mathrm{IC}(A;H)+\mathrm{IC}(B;H)+\mathrm{IC}(C;H).\]
\end{theorem}
\begin{proof}
 The inequality $\max\{\mathrm{IC}(A;H),\mathrm{IC}(B;H),\mathrm{IC}(C;H)\}\leq \mathrm{IC}(G;H)$ follows from Theorem~\ref{thm:inequality-three-graphs}, applied to the inclusions $A\hookrightarrow G$, $B\hookrightarrow G$ and $C\hookrightarrow G$. To establish the other inequality, suppose that $\mathrm{IC}(A;H)=m$, $\mathrm{IC}(B;H)=k$ and $\mathrm{IC}(C;H)=n$. Let $\{f_i:A_i\to H\}_{i=1}^{m}$ be an optimal injective quasi-homomorphism from $A$ to $H$, $\{g_j:B_j\to H\}_{j=1}^{k}$ be an optimal injective quasi-homomorphism from $B$ to $H$ and $\{d_r:C_r\to H\}_{r=1}^{n}$ be an optimal injective quasi-homomorphism from $C$ to $H$. Then, the combined collection $\{f_1:A_1\to H,\ldots,f_m:A_m\to H,g_1:B_1\to H,\ldots,g_k:B_k\to H,d_1:C_1\to H,\ldots,d_n:C_n\to H\}$ is an injective quasi-homomorphism from $G$ to $H$. Consequently, we have $\mathrm{IC}(G;H)\leq m+k+n=\mathrm{IC}(A;H)+\mathrm{IC}(B;H)+\mathrm{IC}(C;H)$.
\end{proof}

Theorem~\ref{thm:category-union} implies the following corollary:

\begin{corollary}\label{cor:linear-complexity}
    Let $G$ and $H$ be groups, and $A$ and $H', H''$ be proper subgroups of $G$ such that $G=A\cup H'\cup H''$. If $\mathrm{IC}(H';H)=\mathrm{IC}(H'';H)=1$, then \[\mathrm{IC}(A;H)\leq \mathrm{IC}(G;H)\leq \mathrm{IC}(A;H)+2.\] 
\end{corollary}

We have the following example.

\begin{example}
 Let $G$ and $H$ be groups such that $\mathrm{IC}(G;H)>1$. Suppose that $\sigma(G)=3$ and $H'$, $H'', H'''$ are proper subgroups of $G$ such that $G=H'\cup H''\cup H'''$ and $\mathrm{IC}(H';H)=\mathrm{IC}(H'';H)=\mathrm{IC}(H''';H)=1$. By Corollary~\ref{cor:linear-complexity} together with Lemma~\ref{thm:lower-bound}, we obtain \[\mathrm{IC}(G;H)=3.\] 
\end{example}

\subsection{Inequality of the product}
Given two groups $G_1$ and $G_2$, the \textit{direct product} $G_1\times G_2$ is considered with the component-wise operation, i.e., $(g_1,g_2)\cdot (g_1',g_2')=(g_1g_1',g_2g_2')$. Given two homomorphisms $f_1:G_1\to H_1$ and  $f_2:G_2\to H_2$, their \textit{direct product} $f_1\times f_2:G_1\times G_2\to H_1\times H_2$ is defined as $(f_1\times f_2)(g_1,g_2)=(f_1(g_1),f_2(g_2))$. This forms a homomorphism from $G_1\times G_2$ to $H_1\times H_2$. Note that if $A$ is a subgroup of $G_1$ and $B$ is a subgroup of $G_2$, then $A\times B$ is a subgroup of $G_1\times G_2$. Furthermore, each coordinate injection $\iota_j:G_j\to G_1\times G_2$ (for $j=1,2$) defined by $\iota_1(g)=(g,1)$ and $\iota_2(g)=(1,g)$ are injective homomorphisms. Hence, we have $\mathrm{IC}(G_j;G_1\times G_2)=1$ for $j=1,2$.

\medskip We have the following statement.

\begin{proposition}[Coordinate Injections]\label{prop:diagonal}
     Let $G$, $G_1$, $G_2$, $H$, $H_1$, and $H_2$ be groups. The following holds:\begin{enumerate}
         \item[(1)] $\mathrm{IC}(G;H_1\times H_2)\leq\min\{\mathrm{IC}(G;H_1),\mathrm{IC}(G;H_2)\}$.
         \item[(2)] $\max\{\mathrm{IC}(G_1;H),\mathrm{IC}(G_2;H)\}\leq \mathrm{IC}(G_1\times G_2;H)$.
     \end{enumerate} 
\end{proposition}
\begin{proof}
 This follows from Theorem~\ref{thm:inequality-three-graphs}, applied to the coordinate injection.
\end{proof}

Proposition~\ref{prop:diagonal} implies the following example.

\begin{example}
  Let $G$ and $H$ be groups. The following holds: \[\mathrm{IC}(G;H\times H)\leq\mathrm{IC}(G;H)\leq  \mathrm{IC}(G\times G;H).\]
\end{example}

The following statement presents the inequality of the product. 

\begin{theorem}[Inequality of the Product]\label{thm:product-inequality-complexity}
    Let $G_1$, $G_2$, $H_1$ and $H_2$ be groups. Then, we have:
    \[\mathrm{IC}(G_1\times G_2;H_1\times H_2)\leq \mathrm{IC}(G_1;H_1)\cdot \mathrm{IC}(G_2;H_2).\]
\end{theorem}
\begin{proof}
  Let $m=\mathrm{IC}(G_1;H_1)$, $n=\mathrm{IC}(G_2;H_2)$, and let $\mathcal{M}_1=\{f_{i,1}:G_{i,1}\to H_1\}_{i=1}^{m}$, and  $\mathcal{M}_2=\{f_{j,2}:G_{j,2}\to H_2\}_{j=1}^{n}$ be optimal injective quasi-homomorphisms from $G_1$ to $H_1$ and from $G_2$ to $H_2$, respectively. The collection $\mathcal{M}_1\times \mathcal{M}_2=\{f_{i,1}\times f_{j,2}:G_{i,1}\times G_{j,2}\to H_1\times H_2\}_{i=1,j=1}^{m,n}$ is an injective quasi-homomorphism from $G_1\times G_2$ to $H_1\times H_2$. Thus, we have $\mathrm{IC}(G_1\times G_2;H_1\times H_2)\leq m\cdot n=\mathrm{IC}(G_1;H_1)\cdot \mathrm{IC}(G_2;H_2)$.  
\end{proof}

We obtain the following example.

\begin{example}
    Let $n\geq 1$ and $p$ be a prime number. By Lemma~\ref{thm:lower-bound} together with Theorem~\ref{thm:product-inequality-complexity}, we obtain \[\sigma\left(C_p^{n+1}\right)\leq\mathrm{IC}\left(C_p^{n+1};C_p^{n}\right)\leq \mathrm{IC}\left(C_p\times C_p;C_p\right).\] On the other hand, $\sigma\left(C_p^{n+1}\right)=p+1$ (by \cite[Theorem 2, p. 45]{cohn1994}, \cite[Theorem, p. 1071]{rosenfeld1963}) and $\mathrm{IC}\left(C_p\times C_p;C_p\right)=p+1$ (by Example~\ref{exam:codomain-zp}). Hence, \[\mathrm{IC}\left(C_p^{n+1};C_p^{n}\right)=p+1.\]
\end{example}

We close this section with the following remark, which presents a direct relation between injective hom-complexity and the sectional number.

\begin{remark}[Injective hom-complexity and sectional number]\label{rem:complexity-sectionalnumber}
Let $G$ and $H$ be groups. Given a homomorphism $f:H\to G$, we have \[\mathrm{IC}(G;H)\leq \text{sec}(f).\] Here, $\text{sec}(f)$ denotes the sectional number of $f$ as introduced in \cite{zapata2024} and developed in \cite{zapata2025}. Specifically, $\text{sec}(f)$ is the least positive integer $k$ such that there exist proper subgroups $G_1\ldots,G_k$ of $G$ with $G=G_1\cup\cdots\cup G_k$, and for each $G_i$, there exists a homomorphism $s_i:G_i\to H$ such that $f\circ s_i=\mathrm{incl}_{G_i}$ (and thus each $s_i:G_i\to H$ is an injective homomorphism), where $\mathrm{incl}_{G_i}:G_i\hookrightarrow G$ is the inclusion homomorphism.  
\end{remark}

\section*{Acknowledgments}
The first author would like to thank grant\#2023/16525-7, and grant\#2022/16695-7, S\~{a}o Paulo Research Foundation (FAPESP) for financial support.  

\section*{Declarations}
\begin{itemize}
\item Conflict of interest/Competing interests: The authors declare that there are no conflicts of interest.
\item Data availability: This manuscript does not report data generation or analysis.
\item Funding: \begin{itemize}
    \item Cesar A. Ipanaque Zapata, Fundação de Amparo à Pesquisa do Estado de São Paulo, 2023/16525-7.
    \item Cesar A. Ipanaque Zapata, Fundação de Amparo à Pesquisa do Estado de São Paulo, 2022/16695-7.
\end{itemize}
\item Author contribution: Cesar A. Ipanaque Zapata, and Martha O. Gonzales Bohorquez wrote, read, and approved the final manuscript.
\end{itemize}


\bibliographystyle{plain}

\end{document}